\documentclass[12pt]{article}

\usepackage{latexsym,amsmath,amsthm,amssymb,amsfonts,amscd}

\usepackage{graphicx}
\usepackage{color}

\usepackage{epsfig}

\usepackage{graphics}

\newtheorem{remark}{Remark}
\newtheorem{theorem}{Theorem}

\newtheorem{definition}{Definition}
\newtheorem{lemma}{Lemma}
\newtheorem*{lemma*}{Lemma}

\begin{document}

\title{The prolific backbone  for supercritical superprocesses.}

\maketitle

\begin{center}
{\large J. Berestycki}\footnote{{\sc Universit\'e Pierre et Marie Curie, Paris, France.} E-mail: julien.berestycki@upmc.fr } \,
{\large  A. E. Kyprianou}\footnote{ {\sc University of Bath, UK..} E-mail: a.kyprianou@bath.ac.uk} \, {\large and}
{\large A. Murillo-Salas}\footnote{ {\sc University of Bath, UK.} E-mail: as454@bath.ac.uk}
 \end{center}
\vspace{0.2in}

\begin{abstract} 
\noindent We develop an idea of Evans and O'Connell \cite{EO},  Engl\"ander and Pinsky \cite{EP} and Duquesne and Winkel \cite{DW} by giving a pathwise construction of the so called `backbone' decomposition for supercritical superprocesses. Our results also complement 
a related result  for critical $(1+\beta)$-superprocesses given in Etheridge and Williams \cite{EW}. Our approach relies heavily on the use of Dynkin-Kuznetsov $\mathbb{N}$-measures. 
\bigskip


\bigskip

\end{abstract}

\vspace{0.5cm}

\section{Introduction}
In Evans and O'Connell \cite{EO}, and later in Engl\"ander and Pinsky \cite{EP}, a new decomposition of a supercritical superprocess with quadratic branching mechanism was introduced in which one may write the distribution of the superprocess at time $t\geq0$ as the result of summing two independent processes together. The first is a copy of the original process conditioned on extinction. The second process is understood as the aggregate accumulation of mass from independent copies of the original process conditioned on extinction which have immigrated `continuously' along the path of an auxilliary dyadic branching particle diffusion which starts with a Poisson number of particles. The embedded branching particle system is known as the {\it backbone} (as opposed to the {\it spine} or {\it immortal particle} which appears in another related decomposition,   introduced in Roelly-Coppeletta and Rouault (1989) and Evans  (1993)). In both \cite{EO} and \cite{EP} the decomposition is seen through the semi-group evolution equations which drive the process semi-group. However no pathwise construction is offered.

In Duquesne and Winkel \cite{DW} a version of this decomposition which, albeit does not take account of spatial motion, was established in much greater generality. In their case, quadratic branching is replaced by a general branching mechanism $\psi$ which is the Laplace exponent of a spectrally positive L\'evy process and which satisfies the conditions
 $0<-\psi'(0+)<\infty$ and 
 $
\int^\infty   1/\psi(\xi)  {\rm d}\xi <\infty. 
$
Moreover, the decomposition is offered in the pathwise sense and described through the growth of genealogical trees embedded within the underling continuous state branching process. 
In their case the backbone is a continuous-time Galton Watson process and the general nature of the branching mechanism induces three different kinds of immigration. Firstly there is continuous immigration which  is described by a Poisson point process of independent processes along the backbone where the rate of immigration is given by a so-called excursion measure which assigns zero initial mass,  and finite life length of the  immigrating processes. A second Poisson point process  along the backbone describes the immigration of independent processes where the rate of immigration is given by the law of the original process conditioned on extinction and with an initial mass randomised by an infinite measure. This accounts for so-called discontinuous immigration. Finally, at the times of branching of the backbone, independent copies of the original process conditioned on extinction are immigrated with randomly distributed initial mass which depends on the number of offspring at the branch point. The last two forms of immigration do not occur when the branching mechanism is purely quadratic. 

Concurrently to the work of \cite{DW} and within the  class of branching mechanisms corresponding to spectrally positive L\'evy processes with paths of unbounded variation (also allowing for the case that $-\psi'(0+)=\infty$), Bertoin et al. \cite{Betal} identify the aforementioned backbone as characterising prolific individuals within the genealogy of the underling continuous state branching process. Here, a prolific individual is understood to be an individual whose descendants become infinite in number.

In this paper we develop the decomposition of Duquesne and Winkel \cite{DW} further by adding in the following additional features. 
We allow the possibility $\int^\infty 1/\psi(\xi){\rm d}\xi =\infty$ which includes the possibility of supercritical processes whose total mass may, with positive probability, die out without this ever happening in a finite time. This also allows the inclusion of branching mechanisms which belong to spectrally positive L\'evy processes of bounded variation (previously excluded in \cite{Betal} and \cite{DW}).  Secondly our decomposition takes care of spatial motion of individuals, thereby bringing the Duquesne-Winkel decomposition back into the setting of superprocesses. Finally, in the case that we ignore spatial motion, our analysis also allows for the case that $-\psi'(0+)=\infty$. Our proof is fundamentally different to that of \cite{DW} and relies largely on the manipulation of the semi-group evolution equations in the spirit of \cite{EO}, taking advantage of the so called $\mathbb{N}$-measure of Dynkin and Kuznetsov \cite{DK}.

The remainder of the paper is structured as follows.  In the next section we introduce some preliminary notation and remind the reader of some standard results relevant to the subsequent exposition. In Section \ref{semigroupsection}  we describe a branching particle diffusion on which independent superprocesses immigrate in three different ways. In particular we give a key result in which the semi-group of the aforementioned process with immigration is characterised. With the latter in hand, we are able to state and prove in Section \ref{main},  the backbone decomposition for supercritical superprocesses. Finally in Section \ref{proof} we give a proof of the the key analytical result in Section \ref{semigroupsection}. Along the way we shall also establish the slightly stronger backbone decomposition for the case of continuous state branching processes (ie. when spatial considerations are ignored).


\section{Preliminaries}

In this section we outline some standard notation and mathematical tools as well as key existing results, all of which will be the ingredients that together will make up the main result.

\subsection{$(\mathcal{P},\psi)$-superprocess}\label{loglaplace}
Suppose that $X=\{X_t: t\geq 0\}$ is any  superprocess motion on $\mathbb{R}^d$ which is well defined for initial configurations in $\mathcal{M}_F(\mathbb{R}^d)$, the space of finite and compactly supported measures, having an associated conservative diffusion semi-group $\mathcal{P}: = \{\mathcal{P}_t: t\geq 0\}$ on $\mathbb{R}^d$ and general branching mechanism $\psi$ taking the form 
\begin{equation*}
\psi(\lambda) = \alpha \lambda + \beta\lambda^2 + \int_{(0,\infty )} (e^{-\lambda x} - 1 + \lambda x \mathbf{1}_{\{x<1\}} )\Pi({\rm d}x)
\end{equation*}
for $\lambda \geq 0$
where $\alpha\in\mathbb{R}$, $\beta\geq 0$ and $\Pi$ is a measure concentrated on $(0,\infty)$ which satisfies $\int_{(0,\infty)}(1\wedge x^2)\Pi({\rm d}x)<\infty$.  This implies that the total mass of the process $X$ is a continuous-state branching process with branching mechanism $\psi$ for which standard references, e.g. \cite{D2, E}, dictate that we need to assume that $-\psi'(0+)<\infty$. Note however that without this condition it is always the case that $-\psi'(0+)\in(-\infty,\infty]$ and within this regime, continuous-state branching processes are always well defined; see for example \cite{G}. To exclude the case of explosive behaviour, we assume throughout that 
\[
\int_{0+}\frac{1}{|\psi(\xi)|}{\rm d}\xi =\infty.
\]
Moreover, we insist that $\psi(\infty)=\infty$ which means that with positive probability the event $\lim_{t\uparrow\infty}||X_t||=0$ will occur; see for example the summary in Chapter 10 of Kyprianou \cite{K}. We refer to such processes throughout as $(\mathcal{P}, \psi)$-superprocesses.

\begin{remark}\rm
It is worthy of note that the assumption that $\mathcal{P}$ is a conservative diffusion semi-group on $\mathbb{R}^d$ can easily be replaced throughout by the much weaker assumption that $\mathcal{P}$ is a general Borel right Markov process with Lusin state space, just as in \cite{EO} or \cite{D1, D2}, at no cost to the analysis. Indeed all of the proofs go through verbatim. However, purely for the sake of presentation, we keep to the more familiar Euclidian setting.
\end{remark}

\begin{remark}\rm
Whilst the vast majority of all literature concerning $(\mathcal{P}, \psi)$-superprocesses requires the branching mechanism satisfies $-\psi'(0+)<\infty$, an example for which an infinite branching rate is permitted can be found in Fleischmann and Sturm \cite{FSt}.
\end{remark}

For each $\mu\in\mathcal{M}_F(\mathbb{R}^d)$,    we denote the law of $X$ with initial configuration $\mu$ by  $\mathbb{P}_\mu$. The following  standard result from the theory of superprocesses (c.f. Dynkin \cite{D1, D2} for example) describes the evolution of $X$ as a Markov process.

\begin{lemma}\label{semigroup}
For all $f\in bp(\mathbb{R}^d)$, the space of non-negative, bounded and measurable functions on $\mathbb{R}^d$,   
\[
 -\log\mathbb{E}_\mu(e^{- \langle f, X_t\rangle}) =  \int_{\mathbb{R}^d }u_f(x, t)\mu({\rm d}x), \, \mu\in\mathcal{M}_F(\mathbb{R}^d), \, t\geq 0
\]  
where  $u_f(x,t)$ is the unique non-negative solution to the integral equation 
\begin{equation}
u_f(x,t )  =\mathcal{P}_t [f](x) - \int_0^t {\rm d}s\cdot \mathcal{P}_s [\psi(u_f(\cdot,t-s))](x).
\label{mild}
\end{equation}
\end{lemma}
Here we have used the standard inner product notation, for $f\in bp(\mathbb{R}^d)$ and  $\mu\in\mathcal{M}(\mathbb{R}^d)$, the space of measures on $\mathbb{R}^d$.
\[
\langle f , \mu\rangle =  \int_{\mathbb{R}^d }f(x)\mu({\rm d}x).
\]
Accordingly we shall write $||\mu|| = \langle 1,\mu \rangle$.

\begin{remark}\rm\label{CSBPsemigrp}
In the case that we take $\mathcal{P}$ to correspond to a particle remaining stationary at a point, equation (\ref{mild}) collapses to the classical integral equation describing the evolution of a continuous state branching process. As alluded to above, it is known in this case that a unique non-negative solution exists, even in the case that $-\psi'(0+)=-\infty$.
\end{remark}

\subsection{Criticality}\label{psi*}

As noted above the total mass of a $(\mathcal{P}, \psi)$-superprocess is a continuous-state branching process with branching mechanism $\psi$. Since there is no interaction between spatial motion and branching we can therefore characterise the $(\mathcal{P}, \psi)$-superprocess into the catagories of supercritical, critical and subcritical accordingly with the same catagories for continuous-state branching processes. Respectively, these cases correspond to $\psi'(0+)<0$, $\psi'(0+)=0$ and $\psi'(0+)>0$. Recall that even when $X$ is supercritical, it is possible that the process becomes {\it extinguished}, i.e. $\lim_{t\uparrow \infty} ||X_t|| =0$. The probability of the latter event is described in terms of the largest root, say $\lambda^*$, of the equation $\psi(\lambda)=0$. Note that it is known (cf. \cite{G}) that $\psi$ is strictly convex with $\psi(0)=0$ and hence since $\psi(\infty)=\infty$ and $\psi'(0+)<0$ it follows that there are exactly two roots in $[0,\infty)$, one of which is $\lambda^*$ and the other is 0. For $\mu\in \mathcal{M}_F(\mathbb{R}^d)$ we have
\begin{equation}
 \mathbb{P}_\mu (\lim_{t\uparrow \infty} ||X_t|| =0) = e^{-\lambda^* ||\mu||}.
 \label{extinguishpr}
\end{equation}
We also recall that, if in addition 
\begin{equation}
\int^\infty \frac{1}{\psi(\xi)} {\rm d}\xi < \infty, 
\label{oo}
\end{equation}
then the event $\{\lim_{t\uparrow \infty} ||X_t|| =0\}$ agrees with the event of 
{\it extinction}, namely $\{\zeta<\infty\}$ where
\[
 \zeta = \inf\{t> 0 : ||X_t || = 0\}.
\]
Moreover, when the integral test in (\ref{oo}) fails, the supercritical continuous-state branching process becomes extinct with zero probability. This means that the event of becoming extinguished corresponds to the total mass trickling away to zero  but none the less being strictly positive at all finite times. An example of an infinite mean supercritical branching mechanism for which the phenomena of becoming extinguished but not extinct is  $\psi(\lambda) = \lambda - \lambda^\alpha$ where $\alpha\in (0,1)$. A second example in this class is Neveu's branching mechansim $\psi(\lambda) = \lambda \log \lambda$. Note that in the first example, the associated spectrally positive L\'evy process has paths of bounded variation as it can be written as the difference of an $\alpha$-stable subordiantor and a linear unit-rate drift. The second example corresponds to a spectrally positive L\'evy process with paths of unbounded variation as in that case it is known that the underlying L\'evy measure is given by $\Pi({\rm d}x) = x^{-2}{\rm d}x$.

\bigskip

\begin{center}
{\bf For the remainder of the paper, unless otherwise stated, we shall henceforth assume only that $\psi$ is a non-exploding, supercritical branching mechanism satisfying $-\psi'(0+)<\infty$. }
\end{center}

\bigskip

It is well known that there is a link between $\psi$ and another branching mechanism $\psi^*$ where, for $\lambda \geq - \lambda^*$,
\begin{eqnarray}
 \psi^*(\lambda) &:=& \psi(\lambda + \lambda^*) \notag\\
&=& \alpha^* \lambda + \beta \lambda^2 + \int_{(0,\infty)}(e^{-\lambda x} - 1+\lambda x\mathbf{1}_{\{x<1\}}) e^{-\lambda^* x}\Pi({\rm d}x)
\label{psistar}
\end{eqnarray}
and
\[
\alpha^* = \alpha + 2\beta\lambda^* + \int_{(0,1)}(1- e^{-\lambda^* x})x\Pi({\rm d}x).
\]
The connection between $\psi$ and $\psi^*$ has a distinct probabilistic interpretation that we shall now briefly discuss. 

Recall that a continuous-state branching process with branching mechanism $\psi$ can always be written as a time-changed  L\'evy process with no negative jumps and whose Laplace exponent is precisely $\psi$; see for example \cite{L, K}. The assumption  $\psi'(0+)<0$ implies that the underlying L\'evy process drifts to $+\infty$. Moreover, by a classical result, the branching mechanism $\psi^*$ corresponds to the underlying L\'evy process conditioned to drift to $-\infty$ (see Exercise 8.1 of \cite{K}). As the next lemma confirms, it also turns out that $\psi^*$ is the branching mechanism of  a superprocess which can be identified as the $(\mathcal{P}, \psi)$-superprocess  conditioned to become extinguished. (Similar results can be found in \cite{Betal}, Abraham and Delmas \cite{AD} and Sheu \cite{S})

\begin{lemma}\label{killedSP}
For each $\mu\in\mathcal{M}_F(\mathbb{R}^d)$, define the law of $X$ with initial configuration $\mu$ conditioned on becoming extinguished by $\mathbb{P}^*_\mu$ (with expectation operator $\mathbb{E}^*_\mu$). Specifically, for all events $A$, measurable in the natural sigma algebra of $X$,  
\[
\mathbb{P}^*_\mu(A)  = \mathbb{P}_\mu(A| \lim_{t\uparrow\infty}|| X_t || =0).
\]
Then, for all uniformly bounded $f:\mathbb{R}^d\rightarrow [-\lambda^*,\infty)$,  
\[
-\log \mathbb{E}^*_{\mu}(e^{-\langle f, X_t\rangle}) = \int_{\mathbb{R}^d}u^*_f(x, t)\mu({\rm d}x)
\]
where 
\begin{equation}
u^*_f(x,t)  =  u_{f+\lambda^*}(x,t) - \lambda^* 
\label{uandu*}
\end{equation} 
and it is the unique solution of 
\begin{equation}
u^*_f(x, t) = \mathcal{P}_t [f](x) - \int_0^t {\rm d}s\cdot \mathcal{P}_s [\psi^*(u^*_f(\cdot,t-s))](x)
\label{u*}
\end{equation}
where $\psi^*(\lambda) = \psi(\lambda + \lambda^*)$ for $\lambda\geq -\lambda^*$. That is to say $(X, \mathbb{P}^*_\mu)$ is a $(\mathcal{P}, \psi^*)$-superprocess.
\end{lemma}
\begin{proof} Define the event $\mathcal{E} = \{\lim_{t\uparrow\infty}|| X_t || =0 \}$.
Making use of the Strong Markov property and (\ref{extinguishpr}), we have  for $f\in bp(\mathbb{R}^d)$,
\begin{eqnarray*}
\mathbb{E}^*_{\mu}(e^{-\langle f, X_t\rangle}  ) &=&\mathbb{E}_{\mu}(e^{-\langle f, X_t\rangle} |  \mathcal{E} )\\
&=& e^{\lambda^* ||\mu||} \mathbb{E}_{\mu}(e^{-\langle f, X_t\rangle} \mathbf{1}_{\mathcal{E}} )\\
&=& e^{\lambda^* ||\mu||} \mathbb{E}_{\mu}(e^{-\langle f, X_t\rangle} \mathbb{P}_{X_t}(\mathcal{E} ))\\
&=& e^{\lambda^* ||\mu||} \mathbb{E}_{\mu}(e^{-\langle f, X_t\rangle} e^{-\lambda^* || X_t  ||})\\
&=& e^{-\langle  u_{f+\lambda^*}(\cdot, t) -  \lambda^*, \mu\rangle }.
\end{eqnarray*}
It is trivial to check that $ u_{f+\lambda^*}(\cdot, t) -  \lambda^*$ solves (\ref{u*}). Moreover,  since $\psi^*$ is the Laplace exponent of a spectrally positive L\'evy process and $\psi^{*\prime}(0+) = \psi'(\lambda^*)>0$, it follows that the solution to (\ref{u*}) is unique by Lemma \ref{semigroup}.
\end{proof}

\subsection{$\mathbb{N}^*$-measure}

Associated to the laws $\{\mathbb{P}^*_{\delta_x}: x\in\mathbb{R}^d\}$ are the  measures $\{\mathbb{N}^*_x: x\in \mathbb{R}^d\}$, defined on the same measurable space,  which satisfy 
\begin{equation}
\mathbb{N}^*_x (1- e^{-\langle f, X_t \rangle}) = -\log \mathbb{E}^*_{\delta_x}(e^{-\langle  f, X_t\rangle})
\label{DK}
\end{equation}
for all $f\in bp(\mathbb{R}^d)$ and $t\geq 0$. Such measures are formally defined and explored in detail in \cite{DK}. The measures $\{\mathbb{N}^*_x:x\in\mathbb{R}^d\}$ will play a crucial role in the forthcoming analysis. Intuitively speaking, the branching property implies that $\mathbb{P}^*_{\delta_x}$ is an infinitely divisible measure on the path space of $X$, $\mathcal{X} : = \mathcal{M}(\mathbb{R}^d)\times[0,\infty)$, and  (\ref{DK})  is a `L\'evy-Khinchine' formula in which  $\mathbb{N}^*_x$ plays the role of its `L\'evy measure'. 
In the context of \cite{DW}, the measure $\mathbb{N}^*$ is the analogue of what Duquesne and Winkel call the {\it excursion measure}, however, whilst the latter encodes genealogical trees, $\mathbb{N}^*$ does not.

\subsection{Prolific individuals}

In Duquesne and Winkel \cite{DW} and Bertoin et al. \cite{Betal} it was shown that there are certain geneaologies embedded in supercritical continous state branching process which are exclusively responsible for the infinite growth  of the process. They show that one may identify such geneaologies in the form of a continuous-time Galton Watson process (that is to say, a version of the Galton Watson process in which individuals remain alive for an independent and exponentially distributed period of time with a common rate before splitting). The generator of such a continuous-time Galton Watson processes is usually identified in the form 
\[
 F(s) = q\sum_{n\geq 0}p_n (s^n -s ),
\]
where $q>0$ is the common rate of splitting and $\{p_n: n\geq 0\}$ is the offspring distribution. For the particular continuous-time Galton Watson process representing the prolific geneaology in a supercritical continuous-state branching process with branching mechanism $\psi$, the aforementioned authors show that 
\begin{equation}
 F(s) = \frac{1}{\lambda^*}\psi(\lambda^*(1-s)), \, s\in(0,1).
 \label{brmech}
\end{equation}
Moreover the individual components of $F$ are given by $q= \psi'(\lambda^*)$, $p_0 = p_1 = 0$ and 
for $n\geq 2$,
\begin{equation}
 p_n = \frac{1}{\lambda^* \psi'(\lambda^*)}\left\{\beta (\lambda^*)^2\mathbf{1}_{\{n=2\}} + (\lambda^*)^n\int_{(0,\infty) } \frac{x^n}{n!} e^{-\lambda^*x} \Pi({\rm d}x)\right\}.
 \label{pn}
\end{equation}

Duquesne and Winkel \cite{DW} go further and show that, when $\psi'(0+)\in(-\infty, 0)$ and (\ref{oo}) holds,  the law of a  continuous state branching processes with branching mechanism $\psi$ is equal to that of a process in which immigration occurs on the  continuous-time Galton Watson process of prolific individuals in three different ways. These are,  two types of Poisson immigration along the life span of each prolific individual  and an additional package of immigration at each point of fission of prolific individuals. In the latter case, if a prolific individual has $n$ prolific offspring then a continuous-state branching process with branching mechanism $\psi^*$ immigrates at that moment of time with random initial mass given by the distribution 
\begin{equation}
\eta_n({\rm d}y) =  \frac{1}{p_n\lambda^* \psi'(\lambda^*)}\left\{\beta (\lambda^*)^2\delta_0({\rm d}y)\mathbf{1}_{\{n=2\}} + (\lambda^*)^n \frac{y^n}{n!} e^{-\lambda^*y} \Pi({\rm d}y)\right\}.
\label{pump-prime}
\end{equation}
In the next section we progress the result of \cite{DW} by relaxing their assumptions on $\psi$ to include the cases that $\psi'(0+)= -\infty$ and $\int^\infty 1/\psi(\xi){\rm d}\xi=\infty$ as well as taking into account spatial considerations at the expense of keeping the condition $-\psi'(0+)\in(0,\infty)$. 

\section{Backbone decomposition}

\subsection{A branching particle diffusion with three types of immigration}\label{semigroupsection}

 Let  $\mathcal{M}_{\rm a}(\mathbb{R}^d)\subset\mathcal{M}_F(\mathbb{R}^d)$  be the space of finite atomic measures on $\mathbb{R}^d$. We shall write  $Z$ for a branching $\mathcal{P}$-motion whose total mass has generator given by (\ref{brmech}). 
 Hence $Z$ is the $\mathcal{M}_{\rm a}(\mathbb{R}^d)$-valued process in which individuals, from the moment of birth, live for an independent and exponentially distributed period of time with parameter $\psi'(\lambda^*)$ during which they execute a $\mathcal{P}$-diffusion issued from their position of birth and at death they give birth at the same position to an independent number of offspring with distribution $\{p_n : n\geq 2\}.$    We shall also refer to $Z$ as the $(\mathcal{P}, F)$-{\it backbone}. 
Its  initial  configuration is denoted by  $\nu\in\mathcal{M}_{\rm a}(\mathbb{R}^d)$. 
Moreover, when referring to individuals in $Z$ we may use of classical Ulam-Harris notation, see for example p290 of Harris and Hardy \cite{HH}. The only feature we really need of the the Ulam-Harris notation is that individuals are uniquely identifiable amongst   $\mathcal{T}$, the set labels of individuals realised in $Z$. For each individual $u\in \mathcal{T}$ we shall write $\tau_u$ and $\sigma_u$ for its birth and death times respectively and  $\{z_u(r): r\in[\tau_u, \sigma_u]\}$ for its spatial trajectory.

Inspired by \cite{DW} and \cite{Betal}, we are interested in  immigrating  independent $(\mathcal{P}, \psi^*)$-superprocesses on   $Z$ in a way that the immigration rate is related to the subordinator 
whose Laplace exponent is given by 
\begin{equation}
\phi(\lambda) = \psi^{*\prime}(\lambda) - \psi^{*\prime}(0) = \psi'(\lambda + \lambda^*) - \psi'(\lambda^*)
\label{isbernstein}
\end{equation}
 together with some additional  immigration at the splitting times of $Z$.  Note in particular that the right hand side of  (\ref{isbernstein}) can be written more explicitly in the form
\[
\phi(\lambda) = 2\beta \lambda + \int_{(0,\infty)}(1- e^{-\lambda x})xe^{-\lambda^* x}\Pi({\rm d }x).
\]

\begin{definition}\label{LZdef}
For $\nu\in\mathcal{M}_{\rm a}(\mathbb{R}^d)$ and $\mu\in\mathcal{M}_F(\mathbb{R}^d)$ let
$Z$ be  a $(\mathcal{P}, F)$-branching diffusion with initial configuration $\nu$ and $\widetilde{X}$ an independent copy of $X$ under $\mathbb{P}^*_\mu$.
Then we define the measure-valued stochastic process $\Lambda = \{\Lambda_t : t\geq 0\}$ on $\mathbb{R}^d$ by 
\[
\Lambda  = \widetilde{X} + I^{\mathbb{N}^*} + I^{\mathbb{P}^*} + I^\eta
\]
where  the processes     $I^{\mathbb{N}^*}= \{I^{\mathbb{N}^*}_t: t\geq 0\}$,  $I^{\mathbb{P}^*}= \{I^{\mathbb{P}^*}_t: t\geq 0\}$ and $I^\eta = \{I^\eta_t : t\geq 0\}$ are  independent of $\widetilde{X}$ and, conditionally on $Z$, are independent of one another. 
Morover, these three processes are described pathwise as follows.

\begin{description}
\item[(i) Continuous immigration:] The process 
$I^{\mathbb{N}^*}$ is measure-valued on $\mathbb{R}^d$ 
such that
\[
 I^{\mathbb{N}^*}_t := \sum_{u\in\mathcal{T}}\sum_{t\wedge \tau_u< r\leq t\wedge \sigma_u}X^{(1,u,r)}_{t-r}
\]
where, given $Z$, independently for each $u\in\mathcal{T}$ such that $\tau_u <t$, the processes $X^{(1,u,r)}_{\cdot}$ are countable in number and correspond to $\mathcal{X}$-valued, Poissonion immigration along the space-time trajectory $\{(z_u(r), r): r\in(\tau_u, t\wedge \sigma_u] \}$ with rate $2\beta {\rm d}r \times {\rm d}\mathbb{N}^*_{z_u(r)}$.

\item[(ii) Discontinuous immigration:] The process
$I^{\mathbb{P}^*}$ is  measure-valued on $\mathbb{R}^d$ 
such that
\[
 I^{\mathbb{P}^*}_t := \sum_{u\in\mathcal{T}}\sum_{t\wedge \tau_u< r\leq t\wedge \sigma_u}X^{(2,u,r)}_{t-r}
\]
where, given $Z$, independently for each $u\in\mathcal{T}$ such that $\tau_u <t$, the processes $X^{(2,u,r)}_{\cdot}$ are countable in number and correspond to $\mathcal{X}$-valued, Poissonion immigration along the space-time trajectory $\{(z_u(r), r): r\in(\tau_u, t\wedge \sigma_u] \}$ with rate
$ {\rm d}r\times \int_{y\in(0,\infty)}ye^{-\lambda^* y}\Pi({\rm d }y)\times {\rm d}\mathbb{P}^*_{y \delta_{z_u(r)}}$.


\item[(iii) Branch point biased immigration:] The process $I^\eta$ is also measure valued  on $\mathbb{R}^d$ such that
\[
 I^\eta_t : = \sum_{u\in\mathcal{T}}\mathbf{1}_{\{\sigma_u\leq t\}}
X^{(3,u)}_{t-\sigma_u}
\]
where, given $Z$, independently for each $u\in\mathcal{T}$ such that $\sigma_u\leq t$, the process $X^{(3,u)}_\cdot$ is an independent copy of $X$ issued at time $\sigma_u$ with law $\mathbb{P}_{Y_u\delta_{z_u(\sigma_u)}}$ where $Y_u$ is an independent random variable with distribution $\eta_{N_u}({\rm d}y)$.

\end{description}
Moreover, we denote the law of $\Lambda$ by $\mathbf{P}_{\mu\times\nu}$.
\end{definition}


\begin{remark}\rm
In the very special case that $\psi(\lambda) = -{\mathtt a} \lambda + {\mathtt b}\lambda^2 $, where ${\mathtt a}, {\mathtt b}>0$, note that the discontinuous and branch point biased immigration are absent. Moreover, $\psi^*(\lambda )  = {\mathtt a }\lambda + {\mathtt b}\lambda^2$, the backbone has binary splitting and  therefore agrees with the backbone in Evans and O'Connell \cite{EO}.
\end{remark}

Note that the total mass $Z_t(\mathbb{R}^d)$ of the backbone is the continuous-time Galton-Watson process  of prolific individuals found in Bertoin et al. \cite{Betal}.
Note also that the process $((\Lambda,Z), \mathbf{P}_{\mu\times\nu})$ is Markovian. This is immediate from three important facts. Firstly  the backbone, $Z$, is a Markov branching diffusion. Secondly, conditional on $Z$ immigrating mass occurs 
independently according to a Poisson point process or as additional indpendent packages at the splitting times of $Z$. Finally, the mass which has immigrated by a fixed time evolves in Markovian way thanks to the branching property. Indeed, using these facts it is not difficult to justify that, for all $s,t\geq 0$
\[
 \mathbf{E}_{\mu\times \nu}(e^{-\langle f, \Lambda_{t+s} \rangle} | \{(\Lambda_u, Z_u): u\leq t\})
= h_s(\Lambda_t, Z_t),
\]
where for $m\in\mathcal{M}(\mathbb{R}^d)$, $n\in\mathcal{M}_{\rm a}(\mathbb{R}^d)$ and $s\geq 0$, $h_s(m,n) = \mathbf{E}_{m\times n}(e^{-\langle f, \Lambda_{s} \rangle} )$.

We conclude with the main result of this section which, amongst other things, shows that $\Lambda$ is a conservative process. The proof is given in section \ref{proof}.

\begin{theorem}\label{PDE} For every $\mu\in\mathcal{M}_F(\mathbb{R}^d)$,  $\nu\in\mathcal{M}_{\rm a}(\mathbb{R}^d)$ and $f,h\in bp(\mathbb{R}^d)$ we have 
\begin{equation}
\mathbf{E}_{\mu\times \nu}(e^{- \langle f, \Lambda_t\rangle - \langle h, Z_t\rangle}) = e^{- \langle u_{f}^*(\cdot , t), \mu\rangle - \langle v_{f,h}(\cdot, t), \nu \rangle} 
\label{randomise this}
\end{equation}
where $\exp\{-v_{f,h}(x,t)\}$ is the unique $[0,1]$-valued solution to the integral equation
\begin{equation}
\label{semi}
e^{-v_{f,h}(x, t)} = \mathcal{P}_t[e^{-h}](x)+\frac{1}{\lambda^*} \int_0^t {\rm d}s\cdot \mathcal{P}_s[\psi^*(-\lambda^* e^{-v_{f,h}(\cdot, t-s)} + u^*_f(\cdot, t-s)) -\psi^*(u^*_f(\cdot, t-s))](x).
\end{equation}
for $x\in\mathbb{R}^d$ and $t\geq 0$. In particular, for each $t\geq 0$,  $\Lambda_t$ has almost surely finite mass.
\end{theorem}

\subsection{Prolific backbone decomposition of a supercritical $(\mathcal{P, \psi})$-superprocess.}\label{main}

A consequence of Theorem \ref{PDE} is the following theorem which constitutes our main result.
It deals with the case that we randomise the law $\mathbf{P}_{\mu\times\nu}$ for $\mu\in\mathcal{M}_F(\mathbb{R}^d)$ by replacing the deterministic choice of $\nu$ with a Poisson random measure having intensity measure $\lambda^*\mu$. We denote the resulting law by $\mathbf{P}_\mu$.

\begin{theorem}\label{main-1}
For any $\mu\in\mathcal{M}_F(\mathbb{R}^d)$, the process $(\Lambda,  \mathbf{P}_\mu)$ is Markovian and has the same law as  $(X,  \mathbb{P}_\mu )$.
\end{theorem}

\begin{proof} The proof is guided by the calculations found in the proof of  Theorem 3.2 of \cite{EO}.
We start by addressing the claim that $(\Lambda, \mathbf{P}_{\mu})$ is a Markov process. Given the Markov property of the pair $(\Lambda,Z)$, it suffices to show that given $\Lambda_t$, the atomic measure $Z_t$ is equal in law to a Poisson random measure with intensity $\lambda^*\Lambda_t$. Thanks to Campbell's formula for Poisson random measures (see e.g. Section 3.2 of \cite{King}), this is equivalent to showing that for all  $h\in bp(\mathbb{R}^d)$, 
\[
 \mathbf{E}_\mu(e^{-\langle h, Z_t\rangle}| \Lambda_t) 
= \exp\{- \langle \lambda^* (1- e^{-h}), \Lambda_t\rangle\},
\]
which in turn is equivalent to showing that for all  $f,h\in bp(\mathbb{R}^d)$,
\begin{equation}
 \mathbf{E}_\mu(e^{- \langle f, \Lambda_t\rangle - \langle h, Z_t\rangle}) = 
\mathbf{E}_\mu(e^{- \langle \lambda^* (1 - e^{-h}) +f, \Lambda_t\rangle}).
\label{LHSRHS}
\end{equation}
Note from (\ref{randomise this}) however that when we randomise $\nu$ so that it has the law of a Poisson random measure with intensity $\lambda^*\mu$, we find the identity
\[
 \mathbf{E}_\mu(e^{- \langle f, \Lambda_t\rangle - \langle h, Z_t\rangle}) =e^ {-\langle u^*_f(\cdot, t) + \lambda^*(1 - e^{- v_{f,h}(\cdot, t)}), \mu \rangle}.
\]
%
Moreover, if we replace  $f$ by  $ \lambda^* (1 - e^{-h}) +f$ and $h$ by $0$ in (\ref{randomise this}) and again randomise $\nu$ so that it has the law of a Poisson random measure with intensity $\lambda^*\mu$ then we get
\[
 \mathbf{E}_\mu(e^{- \langle \lambda^* (1 - e^{-h}) +f, \Lambda_t\rangle}) =
e^{ -\langle u^*_{\lambda^*(1- e^{-h}) + f}(\cdot, t) + \lambda^*(1 - \exp\{- v_{\lambda^*(1- e^{-h}) + f,0}(\cdot, t)\}), \mu\rangle}.
\]
These last two observations indicate that (\ref{LHSRHS})  is  equivalent to showing that for all $f,h\in bp(\mathbb{R}^d)$, $x\in\mathbb{R}^d$ and $t\geq 0$,
\begin{equation}
 u^*_f(x,t) + \lambda^*(1 - e^{- v_{f,h}(x,t)}) = u^*_{\lambda^*(1- e^{-h}) + f}(x,t) + \lambda^*(1 - e^{- v_{\lambda^*(1- e^{-h}) + f,0}(x,t)}). 
\label{needtoprove}
\end{equation}

Note that both left and right hand side of the equality above are necessarily non-negative given they are the Laplace exponents of the left and right hand sides of (\ref{LHSRHS}). Making use of (\ref{u*}) and (\ref{semi}), it is computationally very straightforward to show that both left and right hand side of (\ref{needtoprove}) solve (\ref{mild}) with initial condition $f +\lambda^*(1-e^{-h})$. Since (\ref{mild}) has a unique solution with this initial condition, namely $u_{f + \lambda^*
(1-e^{-h})}(x,t)$, we conclude that (\ref{needtoprove}) holds true. The proof of the claimed Markov property is thus complete.

\medskip

Having now established the Markov property, the proof is complete as soon as we can show that $(\Lambda, \mathbf{P}_{\mu})$ has the same semi-group as $(X, \mathbb{P}_\mu)$. However, from the previous part of the proof we have already established that when $f,h\in bp(\mathbb{R}^d)$,
\[
 \mathbf{E}_\mu(e^{- \langle f, \Lambda_t\rangle - \langle h, Z_t\rangle})  =
 e^{-\langle u_{\lambda^*(1- e^{-h}) + f}, \mu\rangle}
 = \mathbb{E}_\mu(e^{-\langle f + \lambda^*
(1-e^{-h}), X_t\rangle}).
\]
In particular, choosing $h=0$ we find 
\[
\mathbf{E}_\mu(e^{- \langle f, \Lambda_t  \rangle})= 
  \mathbb{E}_\mu(e^{-\langle f , X_t\rangle}) 
\]
which is equivalent to the equality of the semi-groups of $(\Lambda, \mathbf{P}_{\mu})$ and $(X, \mathbb{P}_\mu)$.
\end{proof}

\begin{remark}\rm
In the proof above, we have established the so-called {\it Poissonisation} property of superprocesses and continuous state branching processes. Namely that, when   treating  $\lambda^* X_t$ as an intensity measure of a Poisson random field, one generates a set of points whose positions are equal in law to the support of $Z_t$. Fleischmann and Swart \cite{FS} appeal directly to this idea to analyse the law of $Z$ in terms of $X$. \end{remark}

\begin{remark}\rm
Once the reader is familiar with the main ideas of Theorem \ref{main-1} it should be quite clear how to describe in a pathwise sense the  backbone-type decomposition in Engl\"ander and Pinsky \cite{EP}. In their paper, they work with a spatially dependent branching mechanism $\psi(\cdot, \lambda) = \mathtt{a}(\cdot)\lambda + \mathtt{b}(\cdot)\lambda^2$. Given the semi-group computations in \cite{EP} 
one may easily construct the associated pathwise decomposition. There is no discontinuous immigration and no branch point biased immigration. However continuous immigration does occur along the backbone at rate $2\mathtt{b}(\cdot){\rm d}t\times{\rm d}\mathbb{N}^*_\cdot$
where $\mathbb{N}^*$ is again the measure constructed in \cite{DK} which is related to law of the superprocess conditioned on extinction. The latter, as well as the law of the backbone, are already  described in analytical detail in \cite{EP}.
\end{remark}

\subsection{Prolific backbone decomposition of a supercritical continuous-state branching process}

The analysis leading to the proof of Theorem \ref{PDE} also reveals that the assumption that $-\psi'(0+)<\infty$ can be dropped when considering the backbone decomposition for continuous state branching processes. Formally we state this as a theorem.

\begin{theorem}\label{main-2} When $\mathcal{P}$ corresponds to a particle remaining stationary at a point, say $0$, the conclusion of Theorem \ref{main-1} still holds for all $\mu = x\delta_0$ with $x>0$, even when $-\psi'(0+) = \infty$.
\end{theorem}

\section{Proof of Theorems \ref{PDE} and \ref{main-2}}\label{proof}


To prove Theorem \ref{PDE} it suffices to show, thanks to Lemma \ref{killedSP}, that for all $f, h\in bp(\mathbb{R}^d)$, $\nu\in\mathcal{M}_{\rm a}(\mathbb{R}^d)$ and $t\geq 0$, 
\begin{equation}
\mathbf{E}_{\mu\times \nu} (e^{-\langle f, I_t \rangle - \langle  h, Z_t\rangle }) = e^{-\langle v_{f,h}(\cdot, t), \nu\rangle }
\label{toprove}
\end{equation}
where $I  := I^{\mathbb{N^*}} + I^{\mathbb{P}^*} + I^\eta$ and $v_{f,h}$ solves (\ref{semi}). We do this with the help of some preliminary results.


\begin{lemma}\label{L1}
 For all $f\in bp(\mathbb{R}^d)$, $\nu\in\mathcal{M}_{\rm a}(\mathbb{R}^d)$, $\mu\in\mathcal{M}_F(\mathbb{R}^d)$ and $t\geq 0$, 
 we have 
 \[
 \mathbf{E}_{\mu\times \nu} (e^{-\langle f,  I^{\mathbb{N}^*}_t  + I^{\mathbb{P}^*}_t \rangle}| \{Z_s : s\leq t\}) =
 \exp\left\{ 
-\int_0^t   \langle \phi\circ u_f^*(\cdot, t-s),Z_s\rangle {\rm d}s
\right\}
 \]
\end{lemma}
\begin{proof} 
Using the notation from Definition \ref{LZdef}, write 
\[
\langle f, I^{\mathbb{N}^*}_t  + I^{\mathbb{P}^*}_t\rangle = \sum_{u\in\mathcal{T}}\sum_{t\wedge \tau_u < r \leq t\wedge \sigma_u} \langle
f, X^{(1, u, r)}_{t-r}\rangle  + \sum_{u\in\mathcal{T}}\sum_{t\wedge\tau_u < r \leq t\wedge \sigma_u} \langle
f, X^{(2, u, r)}_{t-r}\rangle .
\]
Hence conditioning on $Z$, appealing to independence of the immigrating processes together with Cambell's formula  and (\ref{DK}) we have 
\begin{eqnarray*}
\lefteqn{ \mathbf{E}_{\mu\times \nu} (e^{-\langle f, I^{\mathbb{N}^*}_t  + I^{\mathbb{P}^*}_t\rangle}   | \{Z_s : s\leq t\}) } && \\
&=&
\exp\left\{
- \sum_{u\in\mathcal{T}}  2\beta  \int_{t\wedge \tau_u}^{t\wedge \sigma_u}{\rm d}r\cdot \mathbb{N}^*_{z_u(r)}(1 - e^{-\langle f,  X_{t-r}\rangle})\right.
\\
&&
\left.\hspace{2cm}
- \sum_{u\in\mathcal{T}} \int_{t\wedge \tau_u}^{t\wedge \sigma_u} \int_{(0,\infty)}{\rm d}r\times ye^{-\lambda^* y} \Pi({\rm d}y) \cdot\mathbb{E}^*_{y\delta_{z_u(r)}}(1 - e^{-\langle f,  X_{t-r}\rangle})
\right\}
\\
&=&
\exp\left\{
- \sum_{u\in\mathcal{T}}  2\beta  \int_{t\wedge \tau_u}^{t\wedge \sigma_u}  {\rm d}r\cdot u_f^*(z_u(r), t-r) \right.
\\
&&
\left.\hspace{2cm}
- \sum_{u\in\mathcal{T}} \int_{t\wedge \tau_u}^{t\wedge \sigma_u}  \int_{(0,\infty)} {\rm d}r\times ye^{-\lambda^* y} \Pi({\rm d}y)\cdot(1 - \exp\{-  u_f^*(z_u(r), t-r) y\})
\right\}
\\
&=& 
\exp\left\{
- \int_0^t {\rm d}r\cdot  \langle \phi\circ u_f^*(\cdot, t-r), Z_r \rangle
\right\}
\end{eqnarray*}
as required. \end{proof}

\begin{lemma}\label{L2}Suppose that  $f,h\in bp(\mathbb{R}^d)$ and 
$g_s(x)$ is jointly measurable in $(s,x)$ and bounded on finite time horizons of $s$. Then for  $x\in\mathbb{R}^d$ and $t\geq 0$,
\[
 \mathbf{E}_{\mu\times \nu} \left(\exp\left\{ 
-\int_0^t   \langle g_{t-s},Z_s\rangle {\rm d}s - \langle f, I^\eta_t\rangle -\langle h, Z_t\rangle
\right\}\right) = e^{-\langle \omega(\cdot, t) ,  \nu\rangle}
\]
where $\exp\{-\omega(x,t)\}$ is the unique $[0,1]$-valued solution to the integral equation
\begin{equation}
e^{-\omega(x,t)} = \mathcal{P}_t[e^{-h}](x) +\frac{1}{\lambda^*} \int_0^t {\rm d}s\cdot \mathcal{P}_{s}[H_{t-s}(\cdot, -\lambda^* e^{-\omega(\cdot, t-s)}) - \lambda^*g_{t-s}(\cdot)e^{-\omega(\cdot, t-s)} ](x)
\label{immigration-semigroup}
\end{equation}
 and, for $\lambda\geq -\lambda^*$, 
\begin{eqnarray*}
H_{t-s}(\cdot, \lambda) &:=&\lambda\psi'(\lambda^*) + \beta \lambda^2  + \int_{(0,\infty)} (e^{-\lambda x} - 1 + \lambda x ) 
e^{-(\lambda^* + u^*_f(\cdot, t-s))x }\Pi({\rm d}x).
\label{H}
\end{eqnarray*}
\end{lemma}
\begin{proof}
Following similar arguments to those in the proof of Theorem 2.2 of Evans and O'Connell \cite{EO}, it suffices to consider the case that, in addition to the assumptions in the statement of the  lemma,  $g$ is time-invariant. Moreover, using the branching property of $Z$ it suffices to consider the case that $\nu = \delta_x$ for $x\in\mathbb{R}^d$. In that case, suppose that $\xi:=\{\xi_t : t\geq 0\}$ is the stochastic process whose semi-group is given by $\mathcal{P}$.  We shall use the expectation operators $\{E_x: x\in\mathbb{R}^d\}$ defined by $E_x(f(\xi_t)) = \mathcal{P}_t[f](x)$. Define a new semi-group (of the diffusion $\xi$ killed at rate $g$)
\[
\mathcal{P}^g_t[f](x)=E_x\left(e^{-\int_0^t g(\xi_s){\rm d}s} f(\xi_t)\right)
\]
for $f,g\in bp(\mathbb{R}^d)$.
Standard Feynman-Kac manipulations (see Lemma 2.3 of \cite{EO}) give us that
\begin{equation}
\mathcal{P}^g_t[f](x)=\mathcal{P}_t[f](x) - \int_0^t {\rm d}s\cdot \mathcal{P}_s[g(\cdot)\mathcal{P}^g_{t-s}[f](\cdot)](x).
\label{FK}
\end{equation}
Conditioning on the time of the first branching and recalling that branching occurs at rate $q=\psi'(\lambda^*)$ we get that
\begin{eqnarray*}
\lefteqn{e^{-\omega(x,t)}}&&\\
&=&e^{-qt}\mathcal{P}_t^g[e^{-h}](x)+q\int_0^t {\rm d}s\cdot e^{-q s}\mathcal{P}^g_s\left[
\sum_{n\geq 2} p_n e^{- n \omega(\cdot, t-s)} \int_{(0,\infty)}\eta_n({\rm d}y) e^{- y u^*_f (\cdot , t-s)}
\right](x).
 \end{eqnarray*}
 Next note from (\ref{pump-prime}) that 
 \begin{eqnarray*}
\lefteqn{ \sum_{n\geq 2} p_n e^{- n \omega(\cdot, t-s)} \int_{(0,\infty)}\eta_n({\rm d}x) e^{- x u^*_f (\cdot , t-s)}}&&\\
&=&\frac{1}{q\lambda^*} \left[ \sum_{n\geq 2}
\beta(\lambda^* e^{-\omega(\cdot, t-s)})^2\mathbf{1}_{\{n = 2\}} + 
\frac{1}{n!} \int_{(0,\infty)}(x\lambda^* e^{-\omega(\cdot, t-s)})^n  e^{-(\lambda^* + u^*_f(\cdot, t-s)) x}\Pi({\rm d}x)
\right]\\
&=&\frac{1}{q\lambda^*} \left[ 
\beta(\lambda^* e^{-\omega(\cdot, t-s)})^2+ 
 \int_{(0,\infty)} (e^{x\lambda^* e^{-\omega(\cdot, t-s)}}  -1 -x\lambda^* e^{-\omega(\cdot, t-s)} )e^{-(\lambda^* + u^*_f(\cdot, t-s)) x}\Pi({\rm d}x)
\right]\\
&=& 
\frac{1}{q\lambda^*} [H_{t-s}(\cdot, - \lambda^* e^{-\omega(\cdot, t-s)}) +q\lambda^* e^{-\omega(\cdot, t-s)}].
 \end{eqnarray*}
We now have that 
\begin{eqnarray}
\lefteqn{e^{-\omega(x,t)}}&&\notag\\
&=&e^{-qt}\mathcal{P}_t^g[e^{-h}](x)+\frac{1}{\lambda^*}\int_0^t {\rm d}s\cdot e^{-q s}\mathcal{P}^g_s \left[
H_{t-s}(\cdot, - \lambda^* e^{-\omega(\cdot, t-s)}) +q\lambda^* e^{-\omega(\cdot, t-s)}
\right](x)\notag\\
&=&\mathcal{P}_t^g[e^{-h}](x)+\frac{1}{\lambda^*}\int_0^t {\rm d}s\cdot \mathcal{P}^g_s \left[
H_{t-s}(\cdot, - \lambda^* e^{-\omega(\cdot, t-s)})
\right](x)
\label{pre-g}
 \end{eqnarray}
where the second equality follows by a standard technique found, for example, in Lemma 4.1.1. of \cite{D2};  see also the computations in \cite{EO}. 

Next, we use (\ref{FK}) and note that
\begin{eqnarray*}
\lefteqn{e^{-\omega(x,t)}} &&\\
&=& \mathcal{P}_t[e^{-h}](x) - \int_0^ t{\rm d}s\cdot \mathcal{P}_s [g(\cdot)\mathcal{P}^g_{t-s}[e^{-h}](\cdot)](x)\\
&&+\frac{1}{\lambda^*}\int_0^t {\rm d}s\cdot\Big\{ \mathcal{P}_s[H_{t-s}(\cdot, -\lambda^* e^{-\omega(\cdot, t-s)})](x)\\
&&\hspace{4cm}
- \int_0^s {\rm d}r\cdot\mathcal{P}_r[g(\cdot)\mathcal{P}_{s-r}^g[H_{t-s}(\cdot, -\lambda^* e^{-\omega(\cdot , t-s)}) ]](x)
\Big\}\\&=& \mathcal{P}_t[e^{-h}](x)+\frac{1}{\lambda^*}\int_0^t {\rm d}s\cdot \mathcal{P}_s[H_{t-s}(\cdot, -\lambda^* e^{-\omega(\cdot, t-s)}) - \lambda^*g(\cdot) e^{-\omega(\cdot, t-s)}](x)
\end{eqnarray*}
where in the final equality we have used (\ref{pre-g}) to deduce that
\begin{eqnarray*}
 \lefteqn{\int_0^ t{\rm d}s\cdot \mathcal{P}_s [g(\cdot)\mathcal{P}^g_{t-s}[e^{-h}](\cdot)](x)+
\frac{1}{\lambda^*}\int_0^ t{\rm d}s\cdot\int_0^s {\rm d}r\cdot\mathcal{P}_r[g(\cdot)\mathcal{P}_{s-r}^g[H_{t-s}(\cdot, -\lambda^* e^{-\omega(\cdot , t-s)}) ]](x)}&&\\
&=&\int_0^ t{\rm d}s\cdot \mathcal{P}_s [g(\cdot)\mathcal{P}^g_{t-s}[e^{-h}](\cdot)](x)+
\frac{1}{\lambda^*}\int_0^ t{\rm d}r\cdot\mathcal{P}_r\left[g(\cdot)\int_r^t {\rm d}s\cdot
\mathcal{P}_{s-r}^g[H_{t-s}(\cdot, -\lambda^* e^{-\omega(\cdot , t-s)})](\cdot)
\right]
(x)\\
&=&\int_0^ t{\rm d}s\cdot \mathcal{P}_s [g(\cdot)\mathcal{P}^g_{t-s}[e^{-h}](\cdot)](x)+
\frac{1}{\lambda^*}\int_0^ t{\rm d}r\cdot\mathcal{P}_r\left[g(\cdot)\int_0^{t-r} {\rm d}\theta\cdot
\mathcal{P}_{\theta}^g[H_{t-\theta-r}(\cdot, -\lambda^* e^{-\omega(\cdot , t-\theta-r)})](\cdot)
\right]
(x)\\
&=& \int_0^ t{\rm d}r\cdot \mathcal{P}_r
\left[
g(\cdot)\left\{\mathcal{P}^g_{t-r}[e^{-h}](\cdot)
+\frac{1}{\lambda^*}\int_0^{t-r} {\rm d}\theta\cdot
\mathcal{P}_{\theta}^g[H_{t-r-\theta}(\cdot, -\lambda^* e^{-\omega(\cdot , t-r-\theta)})](\cdot)\right\}
\right](x) \\
&=& \int_0^ t{\rm d}s\cdot \mathcal{P}_s
\left[g(\cdot)e^{-\omega(\cdot, t-s)}\right](x).
\end{eqnarray*}

The proof is complete as soon as we can establish uniqueness to (\ref{immigration-semigroup}). By multiplying the latter equation  through by $\lambda^*$ we note that, by an application of Lemma 2.1 of \cite{EO} (which offers sufficient conditions for solutions to a general family of integral equations), it has a unique solution providing the assumptions of that lemma are satisfied. For this purpose it suffices to check that for each $y\in\mathbb{R}^d$ and $\lambda\in[0,\lambda^*]$, $J(s, y, \lambda): = [H_s(y,-\lambda)-g(y)\lambda]$ is continuous in $s$ and that for each fixed $T>0$, there exists a $K>0$ such that 
\[
\sup_{s\leq T}\sup_{y\in\mathbb{R}^d}|J (s, y, u(y)) - J(s, y, v(y))|\leq K \sup_{y\in\mathbb{R}^d}|u(y) - v(y)|
\]
where $u$ and $v$ are any two measurable mappings from $\mathbb{R}^d$ to $[0,\lambda^*]$. In light of the assumption of boundedness on $g(y)$, thanks to the triangle inequality, it suffices to check 
that 
for each fixed $T>0$, there exists a $K>0$ such that 
\[
\sup_{s\leq T}\sup_{y\in\mathbb{R}^d}|H_s(y, -u(y)) - H_s(y, -v(y))|\leq K \sup_{y\in\mathbb{R}^d}|u(y) - v(y)|
\]
where $u$ and $v$ are any two measurable mappings from $\mathbb{R}^d$ to $[0,\lambda^*]$.

To this end let us define  for $\lambda\geq -\lambda^*$ and $u\geq 0$,
\[
 \chi_u(\lambda) := \lambda \psi'(\lambda^*)+ \beta \lambda^2 +\int_{(0,\infty)} (e^{-\lambda x} -1 + \lambda x)e^{-(\lambda^*+u)x}\Pi({\rm d}x).
\]
so that by definition
$H_s(y,\lambda) = \chi_{u^*_f(y,s)}(\lambda)$, for $\lambda\geq -\lambda^*$.
We need the following facts about $\chi_u(\lambda)$:
\begin{lemma}\label{L3}
 For $\lambda\geq -\lambda^*$ we have that
\[
\chi_u(\lambda)=\psi^*(\lambda + u) - \psi^*(u)   - \lambda[\psi^{*\prime}(u) - \psi^{*\prime}(0)].
\]
Moreover, when we allow  $\psi'(0+)\in[-\infty,\infty)$, for each $\overline{u}, \underline{u}>0$ we have
\[
 \sup_{\underline{u}\leq u\leq \overline{u}}\sup_{\lambda\in[0,\lambda^*]} |\chi'_u(-\lambda)|<\infty.
\]
If however, $-\psi'(0+)<\infty$ then we may take $\underline{u}=0$ in the above inequality.
\end{lemma}

The proof of this result is somewhat technical and disjoint from the core of the argument that we are currently pursuing, 
so its proof is postponed until the end of the paper.

With the help of the above lemma,  we see that 
for each fixed $T>0$, 
\begin{eqnarray}
\lefteqn{\sup_{s\leq T}\sup_{y\in\mathbb{R}^d}|H_s(y, -u(y)) - H_s(y, -v(y))|}&&\notag\\
&&= \sup_{s\leq T}\sup_{y\in\mathbb{R}^d}|\chi_{u^*_f(y,s)}(-u(y)) - \chi_{u^*_f(y,s)}(-v(y))|\notag\\
&&\leq \sup_{0\leq u^*\leq \overline{u}_T}\sup_{y\in\mathbb{R}^d}|\chi_{u^*}(-u(y)) - \chi_{u^*}(-v(y))|\label{canbeimproved}\\
&&\leq 
K \sup_{y\in\mathbb{R}^d}|u(y) - v(y)|\notag
\end{eqnarray}
where $u$ and $v$ are any two measurable mappings from $\mathbb{R}^d$ to $[0,\lambda^*]$,
\begin{equation}\label{needsfinitederiv}
K=\sup_{0\leq u^*\leq \overline{u}_T}\sup_{\lambda\in[0,\lambda^*]}|\chi'_{u^*}(-\lambda)|<\infty
\end{equation}
(observe that the inequality is true if and only if $\psi'(0+)>-\infty$)
and 
\[
\overline{u}_T = \sup_{s\leq T}\sup_{y\in\mathbb{R}^d}u^*_f(y,s)<\infty.
\]
Note that the finiteness of $\overline{u}_T$ can be deduced as follows. Suppose, without loss of generality, that $f$ is bounded by $\theta\geq 0$. 
Then for all $y\in\mathbb{R}^d$ and $s\geq 0$,
\[
e^{-u_f^*(y, s)} = \mathbb{E}^*_{\delta_y}(e^{-\langle f, X_s \rangle})\geq \mathbb{E}^*_{\delta_y}(e^{-\theta ||X_s||}) = e^{- U_\theta^*(s)}
\]
where $U^*_\theta(s)$ is the unique solution to the equation
\begin{equation}
\label{csb-evolution}
U_\theta^*(s) + \int_0^s \psi^*(U_\theta^*(u)){\rm d}u  = \theta.
\end{equation}
Hence we have $\overline{u}_T \le \sup_{s\leq T}U_\theta^*(s) <\infty$.
\end{proof}

\bigskip

\begin{proof}[Proof of Theorem \ref{PDE}]
Recall from earlier remarks that it suffices to prove  (\ref{toprove}). 
Putting Lemma \ref{L1} and Lemma \ref{L2} together it thus suffices to show that when $g_{t-s}(\cdot) = \phi(u^*_f(\cdot, t-s))$, with  $\phi(\lambda) = \psi^{*\prime}(\lambda) -\psi^{*\prime}(0)$,  we have that $\exp\{-\omega(x,t)\}$ is the unique solution to (\ref{semi}). 
For this to be the case, it is enough that 
\begin{eqnarray}
\nonumber
 \lefteqn{H_{t-s}(\cdot, - \lambda^* e^{-\omega(\cdot , t-s)}) - \lambda^* \phi(u^*_f(\cdot, t-s))e^{-\omega(\cdot, t-s)} }&&\\
&&\hspace{3cm}= \psi^*(- \lambda^*e^{-\omega(\cdot, t-s)} + u^*_f(\cdot, t-s)) -\psi^*(u^*_f(\cdot, t-s)).
\label{finishesproof}
\end{eqnarray}
Note that  in order to appeal to Lemma \ref{L2} above,  we require that $g_s(y)$ is bouned on each finite time horizon of $s$. This follows by virtue of the fact that $\phi$ is a Bernstein function (and therefore concave) and that, as indicated in the proof of Lemma \ref{L2}, for each fixed $T>0$, $0\leq \sup_{s\leq T}\sup_{y\in\mathbb{R}^d}u^*_f(y,s)<\infty$. To prove (\ref{finishesproof}), note that 
\begin{eqnarray*}
\lefteqn{ H_{t-s}(\cdot, -\lambda^* e^{-\omega(\cdot, t-s)} ) - \lambda^*\phi(u^*_f(\cdot, t-s))e^{-\omega(\cdot, t-s)}}&&\\
&& = \chi_{u^*_f(\cdot, t-s)}(-\lambda^* e^{-\omega(\cdot, t-s)}) -
\lambda^* e^{-\omega(\cdot, t-s)}[\psi^{*\prime}(u^*_f(\cdot, t-s))- \psi^{*\prime}(0)]
\end{eqnarray*}
and the desired equality follows by Lemma \ref{L3}.
\end{proof}

Note that up until this point in our reasoning, there are only two points where we have used the assumption that $-\psi'(0+)<\infty$. 
The first place occurs at (\ref{mild}) where classical literature imposes the aforesaid assumption as a sufficient condition to guarantee that a unique non-negative solution exists. The second place occurs is in justifying  the finiteness in (\ref{needsfinitederiv}). See in particular Lemma \ref{L3}. 

Morally speaking the imposition of $-\psi'(0+)<\infty$ in these two cases boils down to the same issue of allowing the application of Gronwall's Lemma to establish the existence of a unique non-negative solution to an integral equation. Moreover, it seems difficult to see how one might remove this condition in general. To see why consider, for example (\ref{needsfinitederiv}). Let $f$ be a compactly supported function and consider all $y$ outside of the support of $f$. It is then clear by definition that  for this $f$ and all such $y,$
$$
u^*_f(y,0)=0,
$$
and hence $\chi_{u^*_f(y,0)}(\lambda)=\chi_0(\lambda) = \psi^*(\lambda)$. If in addition $\psi'(0+)=-\infty$, the failure of the function $\chi_0$ to be Lipschitz on $[-\lambda^*,0]$ prevents us from deducing  (\ref{needsfinitederiv}).

We conclude by reviewing the above arguments when spatial motion is disregarded (formally $\mathcal{P}$ corresponds to a particle remaining stationary at a point) thereby giving the proof of Theorem \ref{main-2}.

\begin{proof}[Proof of Theorem \ref{main-2}]
Suppose now that $-\psi'(0+)=\infty$ and $\mathcal{P}$ corresponds to a particle remaining stationary at a point. It suffices to show that the two points (noted in the discussion above) where the condition $-\psi'(0+)<\infty$ was used no longer need this assumption.

With regard to the use of (\ref{mild}), recall that, in the current setting where $X_t = ||X_t||$ for all $t\geq 0$, (\ref{mild}) collapses to the integral equation (\ref{csb-evolution}). Moreover,  as alluded to in the discussion at the beginning of Section  \ref{loglaplace}, (\ref{csb-evolution}) always has a unique non-negative solution even when $-\psi'(0+)=\infty$; see also Remark \ref{CSBPsemigrp}.

With regard to justifying  the finiteness in (\ref{needsfinitederiv}), note that in the current setting, the quantity $u^*_f(y,s)$ can be replaced by $U^*_\theta(s)$ for $\theta\geq 0$. Since  $\underline{u}_T: = \inf_{s\leq T}U^*_\theta(s)> 0$ then the estimate in (\ref{canbeimproved}) can be replaced by 
\[
\sup_{\underline{u}_T\leq u^*\leq \overline{u}_T}\sup_{y\in\mathbb{R}^d}|\chi_{u^*}(-u(y)) - \chi_{u^*}(-v(y))|.
\]
Now proceeding with the proof of Lemma \ref{L2}, taking note of the conclusion of Lemma \ref{L3}, we see that the condition $-\psi(0+)<\infty$ is no longer necessary.

\end{proof}

\subsection{Proof of Lemma \ref{L3}}



 It is a strightforward algabraic exercise to deduce that 
\begin{eqnarray*}
  \lefteqn{\psi^*(\lambda + u) - \psi^*(u) }&&\\
&=& \lambda\left( 2\beta u + \alpha^* + \int_{(0,1)} (1- e^{-ux})x e^{-\lambda^* x}\Pi({\rm d}x)\right) \\
&&+\beta \lambda^2 + \int_{(0,\infty)}(e^{-\lambda x}-1 +\lambda x \mathbf{1}_{\{x<1\}})e^{-(\lambda^* +u)x}\Pi({\rm d}x).
\end{eqnarray*}
It follows, with the help of (\ref{psistar}), that 
\begin{eqnarray*}
 \lefteqn{\psi^*(\lambda + u) - \psi^*(u) }&&\\
&=&\chi_u(\lambda) -\lambda \psi'(\lambda^*) - \lambda\int_{[1,\infty)}x e^{-(\lambda^* + u)x}\Pi({\rm d}x)\\
&&+\lambda\left(2\beta u + \alpha^* + \int_{(0,1)} (1- e^{-ux})x e^{-\lambda^* x}\Pi({\rm d}x)\right)\\
&=&\chi_u(\lambda) +\lambda\left(2\beta u + \alpha^* + \int_{(0,\infty)} (\mathbf{1}_{\{x<1\}}- e^{-ux})x e^{-\lambda^* x}\Pi({\rm d}x) - \psi^{*\prime}(0)\right)
\end{eqnarray*}
However, it is again a simple exercise to deduce from (\ref{psistar})  that 
\[
\psi^{*\prime}(u) = 2\beta u + \alpha^* + \int_{(0,\infty)} (\mathbf{1}_{\{x<1\}}- e^{-ux})x e^{-\lambda^* x}\Pi({\rm d}x)
\]
and hence the first part of the lemma follows.

Next notice that $\chi_u(\lambda)$ is the Laplace exponent of a spectrally positive L\'evy process and therefore strictly convex and infinitely smooth on $(-\lambda^*,\infty)$. Moreover, remembering that $\psi^*(\lambda) = \psi(\lambda+ \lambda^*)$, we have for all $\lambda\geq -\lambda^*$, 
\[
 \chi'_u(\lambda) = \psi'(\lambda+\lambda^*+u) - \psi'(\lambda^*+u)+\psi'(\lambda^*).
\]
The proof of the remaining parts of the lemma are now straightforward. \hfill$\square$

\section*{Acknowledgment} 

A.E.K. would like to thank  Thomas Duquesne for an enlightening discussion. All three authors are grateful to two anonymous referees for their comments. In particular, one referee read the paper meticulously and accordingly alerted us to many improvements of the original version.   J.B. and A.E.K.  acknowledge the support of Royal Society Grant nr. IV0871854 and the Bath Institute for Complex Systems, J.B. is Supported by the \emph{Agence Nationale de la
Recherche} grants ANR-08-BLAN-0220-01 and ANR-08-BLAN-0190 and A.M-S. acknowledges the support of CONACyT grant number 000000000093984.

\end{document}